\documentclass[11pt]{amsart}
\usepackage{amsmath,amssymb,latexsym,dsfont}
\usepackage[small]{caption}
\usepackage{graphicx,color,mathrsfs,tikz}
\usepackage{subfigure,color}
\usepackage{cite}
\usepackage[colorlinks=true,urlcolor=blue,
citecolor=red,linkcolor=blue,linktocpage,pdfpagelabels,
bookmarksnumbered,bookmarksopen]{hyperref}
\usepackage[italian,english]{babel}
\usepackage{units}
\usepackage{enumitem}
\usepackage[left=2.6cm,right=2.6cm,top=2.7cm,bottom=2.7cm]{geometry}
\usepackage[hyperpageref]{backref}

\newtheorem{lemma}{Lemma}[section]
\newtheorem{corollary}[lemma]{Corollary}
\newtheorem{proposition}[lemma]{Proposition}
\newtheorem{theorem}[lemma]{Theorem}

\theoremstyle{definition}

\newtheorem{remark}[lemma]{Remark}
\newtheorem{open}[lemma]{Open problem}

\numberwithin{equation}{section}

\renewcommand{\epsilon}{\eps}
\renewcommand{\i}{{\rm i}}

\renewcommand{\H}{{\mathcal H}}
\renewcommand{\L}{{\mathscr L}}

\newcommand{\C}{{\mathbb C}}
\newcommand{\N}{{\mathbb N}}
\newcommand{\R}{{\mathbb R}}
\newcommand{\mS}{{\mathbb S}}

\newcommand{\eps}{\varepsilon}

\newcommand{\pnorm}[2][]{\if #1'' \left|#2\right|_p \else \left|#2\right|_{#1} \fi}

\newcommand{\loc}{{\rm loc}}

\renewcommand{\theta}{\vartheta}

\title[{Remarks on rearrangement for nonlocal functionals}]{Some remarks on rearrangement \\ for nonlocal functionals}

\author[H.-M. Nguyen]{Hoai-Minh Nguyen}
\author[M.\ Squassina]{Marco Squassina}

\address[H.-M. Nguyen]{Department of Mathematics \newline\indent
	EPFL SB CAMA \newline\indent
	Station 8 CH-1015 Lausanne, Switzerland}
\email{hoai-minh.nguyen@epfl.ch}

\address[M.\ Squassina]{Dipartimento di Matematica e Fisica \newline\indent
	Universit\`a Cattolica del Sacro Cuore \newline\indent
	Via dei Musei 41, I-25121 Brescia, Italy}
\email{marco.squassina@unicatt.it}

\thanks{The second author is member of {\em Gruppo Nazionale per l'Analisi Ma\-te\-ma\-ti\-ca, la Probabilit\`a e le loro Applicazioni} (GNAMPA) of the {\em Istituto Nazionale di Alta Matematica} (INdAM)}



\begin{document}

\begin{abstract}
We prove that a nonlocal functional approximating the standard Dirichlet $p$-norm fails to decrease under two-point rearrangement. Furthermore, we get other properties related to this functional  such as decay and compactness.
\end{abstract}
\maketitle
	

\section{Introduction}

The well-known Polya-Szeg\"o inequality \cite{polya1,polya2,bzim} states that if $u\in W^{1,p}(\R^N,\R^+)$ then 
\begin{equation}
\label{PZ}
\int_{\R^N}|\nabla u^*|^p dx\leq \int_{\R^N}|\nabla u|^p dx.
\end{equation}
where $u^*$ is the Schwarz symmetric rearrangement of $u$. This inequality has relevant applications in the study of isoperimetric inequalities, in the 
Faber-Krahn inequality  and in the determination of optimal constants in the Sobolev inequality \cite{aubin,talenti}. This kind of inequalities still holds in the nonlocal case, e.g.\ for the standard fractional norm, namely for $u\in W^{s,p}(\R^N,\R^+)$
\begin{equation}
\label{FPZ}
\iint_{\R^{2N}} \frac{|u^*(x)-u^*(y)|^p}{|x-y|^{N+ps}} dxdy \leq \iint_{\R^{2N}} \frac{|u(x)-u(y)|^p}{|x-y|^{N+ps}} dxdy,
\end{equation}
for $p\geq 1$ and $s\in (0,1)$, see e.g.\ \cite{almgrenlieb,bae}. Actually this inequality 
implies \eqref{PZ} by a straightforward application of a result by Bourgain, Brezis and Mironescu \cite{bourg,bourg2} which confirms that 
\begin{equation}\label{BBM}
\lim_{s \nearrow 1} (1-s)\iint_{\R^{2N}} \frac{|v(x)-v(y)|^p}{|x-y|^{N+ps}} dxdy = K_{N, p}  \int_{\R^N} |\nabla v|^p,\quad  \mbox{ for } v \in W^{1, p}(\R^N),  
\end{equation}
where 
\begin{equation}\label{KNp}
K_{N, p} : = \int_{\mS^{N-1}} |\boldsymbol{e} \cdot \sigma|^p \, d \sigma,\quad 
\mbox{ for some } \boldsymbol{e} \in \mS^{N-1}, \mbox{ the unit sphere in $\R^N$}. 
\end{equation}
Polarization by closed half spaces $H\subset\R^N$ containing the origin is an elementary form of symmetrization and it is a key tool in order to investigate various rearrangements inequalities.  The polarization $u^H$ with respect to $H$, see the definition \eqref{polar-def}, also called two-point rearrangement, essentially compares the values of $u$ on the two sides of $\partial H$ and keeps the largest values inside $H$ and the smallest values outside $H$, cf.\
\cite{dub1,dub2,dub3,brock}.
Since the first achievements obtained in \cite{brock}, the approximation in $L^p(\R^N)$ of $u^*$ via iterated polarizations of $u$ has been refined in various ways. It is now known that there exists an explicit and universal (i.e.\ independent of $u$) sequence of closed half spaces $\{H_n\}_{n\in\N}$ of $\R^N$ containing the origin such that 
a suitable sequence of iterated polarizations of $u$ with respect to $H_n$
strongly converges to $u^*$ in $L^p(\R^N)$, see \cite{jean} and the references therein. 
It is thus natural to derive the rearrangement inequalities \eqref{PZ}-\eqref{FPZ} from 
 the (possible) corresponding inequalities for the polarizations using general weak lower semi-continuity properties. In fact, for any closed half space $H$ with $0\in H$ and any $u\in W^{1,p}(\R^N)$, 
\begin{equation}
\label{ide-case}
\int_{\R^N}|\nabla u^H|^p dx= \int_{\R^N}|\nabla u|^p dx,
\end{equation}
as well as, for any $W^{s,p}(\R^N),$
\begin{equation}
\label{frak-in}
\iint_{\R^{2N}} \frac{|u^H(x)-u^H(y)|^p}{|x-y|^{N+ps}} dxdy \leq \iint_{\R^{2N}} \frac{|u(x)-u(y)|^p}{|x-y|^{N+ps}} dxdy,
\end{equation}
see  \cite{bae,brock} and references included. For applications of polarization techniques, see
\cite{london,royal-squa,pucci1,jean,jean0,linceisqua}.

More recently, a new class of nonlocal functionals has been involved in the study of topological degree of a map \cite{degree1,degree2,degree3}, namely, for $\delta>0$ and 
$1<p<N$,
$$
I_\delta(u):=\iint_{\{|u(y)-u(x)|>\delta\}}\frac{\delta^p}{|x-y|^{N+p}}dxdy.
$$
It turns out that this energy also provides a pointwise approximation 
of $\|\nabla u\|_p^p$ for $p>1$, precisely 
\begin{equation}
\label{Ng-form}
\lim_{\delta\searrow 0} I_\delta(u)=\frac{1}{p} K_{N, p} \int_{\R^N}|\nabla u|^p dx,
\end{equation}
where $K_{N, p}$ is given by \eqref{KNp}, see  \cite{nguyen06,BourNg} (and also \cite{bre,BHN,BHN2,BHN3,NgSob2}). Various properties of Sobolev spaces in terms of $I_\delta$ were investigated in \cite{nguyen11}, for example, it was shown in \cite[Theorem 3]{nguyen11} that, for $1 < p < N$ and for $\delta > 0$,  
\begin{equation}\label{Sobolev}
\left(\int_{\{|u|>\lambda\delta\}}|u|^{\frac{Np}{N-p}}dx\right)^{\frac{N-p}{Np}}\leq C \big(I_\delta(u) \big)^{1/p}, \quad \forall u \in L^p(\R^N), 
\end{equation}
for some positive constants $C,\lambda$ depending only on $N$ and $p$.

The first goal of this note is to prove that, 
however, the nonlocal energy $I_\delta$ {\em fails} to
be decreasing upon polarization. This supports yet again the idea (cf.\ \cite{NgGamma}) that $I_\delta$ is a much more delicate approximation of local norm with respect to the other mentioned above. More precisely, we have 

\begin{theorem}\label{thm-example}
Let $N \ge 1$ and $p \ge 1$. Then there exist $\delta_0>0$ and  a closed half space $H\subset \R^N$ with $0\in H$ such that for any  $0 < \delta < \delta_0$ there exists a measurable function $u:\R^N\to\R$ such that $I_\delta(u^H)>I_\delta(u)$.
\end{theorem}

The proof of Theorem~\ref{thm-example} is given in Section~\ref{sect-example}.  In Section~\ref{sect-FPZ}, we present a proof of \eqref{FPZ}. 

\smallskip 
The second goal of this note is to prove some facts related to symmetric functions. Precisely:

\vskip4pt
\noindent
$\bullet$ ({\em Global compactness and $I_\delta$}). 
Let $1<p<N$, $(\delta_n) \to 0_+$,  and let $\{u_n\}_{n\in\N}$ be a sequence of radially symmetric decreasing functions. Assume that 
$$
\{u_n\}_{n \in \N} \mbox{ is bounded in } L^p(\R^N) \mbox{ and } \{ I_{\delta_n}(u_n) \}_{n \in \N} \mbox{ is bounded in $\R$}. 
$$
Then $\{ u_n\}_{n\in\N}$ is pre-compact in  $L^r(\R^N)$ for every $p<r<Np/(N-p)$.
\vskip4pt
\noindent
$\bullet$ ({\em Decay and integrability of $I_\delta$}). If $u:\R^N\to\R^+$
and there exists $\vartheta>0$ with
\begin{equation*}
	\int_0^\infty \frac{I_\delta(u^*)^{{\vartheta/p}}}{\delta} d\delta<\infty,
\end{equation*}
then there exists $C>0$ depending on $u,N,p,\vartheta$ such that
$0\le u^*(x)\le C |x|^{-(N-p)/p}.$


\vskip3pt
\noindent
The proof of these two facts is given in Section~\ref{sect-compact} (see 
Theorems \ref{decay} and \ref{globalc}).



\section{Symmetrization inequalities}

\subsection{The defect of decreasingness of $I_\delta$}
In the following $H$ will denote a closed half-space of $\R^N$ containing the origin.
We denote by ${\mathcal H}$ the set of these closed half-spaces. 
A reflection $\sigma_H:\R^N\to\R^N$ with respect to $H$
is an isometry such that $\sigma_H^2={\rm Id}$ and 
$|x-y|<|x-\sigma_H(y)|$ for all $x,y\in H$.
We also set $H^c:=\R^N\setminus H$. Given $x\in\R^N$, $\sigma_H(x)$ 
will also be denoted by $x^H$.
	The polarization (or two-point rearrangement) of a nonnegative real valued function
	$u:\R^N\to\R^+$ with respect to a given $H$ is defined as 
	\begin{equation}
	\label{polar-def}
	u^H(x):=
	\begin{cases}
	\max\{u(x),u(\sigma_H(x))\}, & \text{for $x\in H$}, \\
	\min\{u(x),u(\sigma_H(x))\}, & \text{for $x\in H^c$}.
	\end{cases}
	\end{equation}

\noindent
Let us set 
\begin{equation*}
A:= \big\{x \in \mbox{int}(H): u(x) \le u(x^H) \big\}, \quad B :=  \big\{x \in \mbox{int}(H): u(x) >u(x^H) \big\},
\end{equation*}
\begin{equation*}
C := \big\{y \in H^c: u(y) \ge u(y^H) \big\}, \quad \mbox{ and } \quad D :=  \big\{y \in H^c: u(y) < u(y^H) \big\}.
\end{equation*}
It is clear that 
$$
C = \sigma(A) \quad \mbox{ and } \quad D = \sigma(B). 
$$
We have 
\begin{equation}\label{I-decomposition}
I_\delta = I_\delta^{(A \cup C) (A \cup C)} + I_\delta^{(B \cup D) (B \cup D)}+ I_\delta^{(A\cup C)(B \cup D)} + I_\delta^{(B \cup D) (A \cup C)},  
\end{equation}
where, for two measurable subsets $O, \, P $ of $\R^N$, we denote 
$$
I_\delta^{OP}(u) : = \mathop{\int_O \int_P}_{\{|u(y)-u(x)|>\delta\}}\frac{\delta^p}{|x-y|^{N+p}}dxdy. 
$$
We claim that 
\begin{equation}\label{Decomposition-1}
I_\delta^{(A \cup C) (A \cup C)} (u^H) = I_\delta^{(A \cup C) (A \cup C)} (u), 
\end{equation}
\begin{equation}\label{Decomposition-2}
I_\delta^{(B \cup D) (B \cup D)}(u^H) = I_\delta^{(B \cup D) (B \cup D)}(u).  
\end{equation}
We begin with \eqref{Decomposition-1}. Since 
$$
u^H(x) = u(x^H) \mbox{ if } x \in A \cup C \quad \mbox{ and } \quad |x -y| = |x^H - y^H|, 
$$
by a change of variable $x = x^H$ and $y = y^H$, we obtain 
$$
I_\delta^{(A \cup C) (A \cup C)}(u^H) = I_\delta^{(A \cup C) (A \cup C)}(u), 
$$
which is \eqref{Decomposition-1}.  We next establish \eqref{Decomposition-2}. This is a direct consequence of the fact
$$
u^H(x) = u(x),\quad \mbox{if $x \in B \cup D$}. 
$$
We next concern about the validity of the inequality:
\begin{equation}\label{Decomposition-3}
I_\delta^{(A\cup C)(B \cup D)}(u^H)  \le I_\delta^{(A\cup C)(B \cup D)}(u). 
\end{equation}
By a  change of variables, we obtain 
\begin{equation*}
I_\delta^{(A\cup C)(B \cup D)}(v) 
= \int_A \int_B  \left(\frac{\L(v, x, y)}{|x-y|^{N+p}} + \frac{\L(v, x, y^H)}{|x - y^H|^{N+p}} + \frac{\L(v, x^H, y)}{|x^H - y|^{N+p}} + \frac{\L(v, x^H, y^H)}{|x^H - y^H|^{N+p}}  \right) \, dx \, dy, 
\end{equation*}
where 
$$
\L(v, x, y) := \delta^p\, \mathds{1}_{\{|v(x) - v(y)|>\delta\}}(x, y),
\,\,\quad x,y\in \R^N.
$$
We have, for $x \in A$ and $y \in B$,  
$$
u^H(x) = u(x^H), \quad u^H(x^H) = u(x), \quad u^H(y) = u(y), \mbox{ and } \quad u^H(y^H) = u(y^H). 
$$
It follows that, for $x \in A$ and $y \in B$,  
\begin{equation*}
\frac{\L(u^H, x, y)}{|x-y|^{N+p}} = \frac{\L(u, x^H, y)}{|x-y|^{N+p}}, \quad  \frac{\L(u^H, x, y^H)}{|x - y^H|^{N+p}} = \frac{\L(u, x^H, y^H)}{|x - y^H|^{N+p}}
\end{equation*}
and
\begin{equation*}
\frac{\L(u^H, x^H, y)}{|x^H-y|^{N+p}} = \frac{\L(u, x, y)}{|x^H-y|^{N+p}}, \quad  \frac{\L(u^H, x^H, y^H)}{|x^H - y^H|^{N+p}} = \frac{\L(u, x, y^H)}{|x - y|^{N+p}}.
\end{equation*}
Setting
\begin{equation*}
{\mathscr D}_\delta^H(u, x,y):=\L(u, x^H, y) + \L (u, x, y^H) - \L(u, x, y) - \L(u, x^H, y^H),
\end{equation*}
we derive that inequality~\eqref{Decomposition-3} is equivalent to 
\begin{equation}
\label{Decomposition-3-1}
\mathfrak{d}_\delta^H(u):=\int_A \int_B {\mathscr D}_\delta^H(u, x,y)
\Big(\frac{1}{|x-y|^{N+p}} - \frac{1}{|x^H- y|^{N+p}} \Big) \, dx \, dy \leq 0. 
\end{equation}
On the other hand, in general, concrete examples show that the inequality ${\mathscr D}_u(x,y)$ fails, so that it is expected that $I_\delta(u^H)-I_\delta(u)$ can be positive for some $H,u$ and $\delta>0$, in which case the quantity $\mathfrak{d}_\delta^H(u)$ provides
a measure of the defect of decreasingness.

\begin{remark}[Vanishing defect]\rm 
For any closed half space $H\subset \R^N$ with $0\in H$ there holds, for $p>1$, 
$$
\lim_{\delta\searrow 0} \mathfrak{d}_\delta^H(u)=0,\quad \text{for all $u\in W^{1,p}(\R^N)$}.
$$
In fact, we know that $u^H\in W^{1,p}(\R^N)$ also and $\|\nabla u^H\|_{L^p(\R^N)}=\|\nabla u\|_{L^p(\R^N)}$ from formula \eqref{ide-case}. Furthermore, from formulas \eqref{I-decomposition}-\eqref{Decomposition-3}, we infer
$$
 \mathfrak{d}_\delta^H(u)=\frac{1}{2}I_\delta(u^H)-\frac{1}{2} I_\delta(u),\quad
 \text{for all $\delta>0$}.
$$
Then, from equality \eqref{Ng-form}, we conclude
$$
\lim_{\delta\searrow 0} \mathfrak{d}_\delta^H(u)=
\frac{1}{2}\lim_{\delta\searrow 0} I_\delta(u^H)-\frac{1}{2}\lim_{\delta\searrow 0} I_\delta(u)=
\frac{K_{N, p}}{2p}\big(\|\nabla u^H\|_{L^p(\R^N)}^p-\|\nabla u\|_{L^p(\R^N)}^p\big)=0,
$$
proving the assertion.
\end{remark}

\subsection{Proof of Theorem~\ref{thm-example}}
\label{sect-example}

We first deal with the case $N=1$.  Here is a counterexample to \eqref{Decomposition-3-1} with $p \in [1, + \infty)$.  Fix $\eps \in (0, 1/8)$ and 
let $u: \R \to \R$ be defined by
\begin{equation}\label{MM-1}
u(x) = \left\{ \begin{array}{cl} \delta & \mbox{ for } x \in (-2, -1], \\[6pt]
\mbox{ linear } & \mbox{ for } x \in [-1, -1 + \delta], \\[6pt] 
- 2 \eps \delta & \mbox{ for } x \in [-1 + \delta, 0), \\[6pt]
-\eps \delta  & \mbox{ for } x \in (0, 1), \\[6pt]
\delta - \eps \delta & \mbox{ for } x \in (1, 2),\\[6pt]
0 & \mbox{ for } x \not \in (-2, 2). 
\end{array}\right.
\end{equation}
Let $H = [0, \infty)$ and $\sigma$ be the standard reflection.  
It is clear that $u^H$ satisfies
\begin{equation}\label{MM-2}
u^H(x) = \left\{ \begin{array}{cl} \delta - \eps \delta & \mbox{ for } x \in (-2, -1], \\[6pt]
- 2 \eps \delta  & \mbox{ for } x \in [-1 + \delta, 0), \\[6pt]
-\eps \delta & \mbox{ for } x \in (0, 1-\delta), \\[6pt]
\delta & \mbox{ for } x \in (1, 2), \\[6pt]
0 & \mbox{ for } x \not  \in (-2, 2). 
\end{array}\right.
\end{equation}
We derive from \eqref{MM-1} and \eqref{MM-2} that 
\begin{multline*}
I_\delta(u^H) - I_{\delta} (u) \ge \int_0^{1-\delta} \int_1^2 \frac{\delta^p}{|x-y|^{p+1}} \, dx \, dy \\[6pt] - 2 \mathop{\int_{-1}^{-1 + \delta} \int_{-2}^0}_{\{|u(x) - u(y)| > \delta\}} \frac{\delta^p}{|x - y|^{p+1}} \, dx \, dy - 2 \int_{-1}^{0} \int_{1}^2\frac{\delta^p}{|x - y|^{p+1}} \, dx \, dy. 
\end{multline*}
A straightforward computation yields, for small $\delta$ and $\eps$,  
\begin{equation*}
\int_0^{1-\delta} \int_1^2 \frac{\delta^p}{|x - y|^{p+1}} \, dx \, dy \simeq \left\{\begin{array}{cl}
\delta |\ln \delta| &  \mbox{ if } p = 1 \\[6pt]
\delta  & \mbox{ if } p>1, 
\end{array}\right.
\end{equation*}
\begin{equation*}
\mathop{\int_{-1}^{-1 + \delta} \int_{-2}^0}_{\{|u(x) - u(y)| > \delta\}} \frac{\delta^p}{|x - y|^{p+1}} \, dx \, dy \simeq \eps \delta, 
\end{equation*}
and 
\begin{equation*}
\int_{-1}^{-0} \int_{1}^2\frac{\delta^p}{|x - y|^{p+1}} \, dx \, dy \simeq \delta^p. 
\end{equation*}
We obtain that, for small positive $\delta$ and $\eps$, we have $I_\delta(u^H) > I_{\delta}(u)$.

\medskip This example can be modified to obtain similar conclusion in the case $\{ 0 \} \in H$ by considering the function $u( \cdot + c)$ for some $c>0$ where $u$ is given above.  In the above example, the function $u$ is not non-negative. However, this point can be handled by considering the function given by  $ u(x) + 2 \eps \delta$ if $|x| < 3$ and 0 otherwise. 

\medskip 
We next consider the case $N \ge 2$. Set 
$$
H = [0, +\infty) \times \R^{N-1}
$$ 
and define $U: \R^N \to \R$ as follows, for $(x_1, x') \in \R \times \R^{N-1}$,  
\begin{equation*}
U(x_1, x') = \left\{ \begin{array}{cl} u(x_1) - \delta/2 & \mbox{ if } |x|_\infty \le 2, \\[6pt]
0 &  \mbox{ otherwise}, 
\end{array}\right.
\end{equation*}
where $u$ is given in \eqref{MM-1}. 
One can check that 
\begin{equation*}
I_\delta(U^H) - I_\delta(U) \simeq I_\delta (u^H) - I_\delta(u) > 0. 
\end{equation*}

In the above example, the function $U$ is not non-negative and $H$ does not contain the origin. However, this point can be handled similarly as in the case $N=1$. 

\medskip
The following question remains open: 

\begin{open} Let $N \ge 1$. 
It is true that $I_\delta(u^*)\leq I_\delta(u)$ for any 
measurable $u:\R^N\to\R^+$ and $\delta>0$?
\end{open}

\subsection{Riesz two-point inequality}
	\label{sect-FPZ} 
	Let $u \in L^1_{\loc}(\R^N)$, let $H$ be a closed half-space of $\R^N$
	and let $G$ be a Young function, i.e.\ $G: [0, + \infty) \to \R$, 
	$G(0) = 0$, $G$ is non-decreasing, and $G$ is convex,
	and let $w$ be a non-negative, non-increasing radial function. 
	The above notations allow to prove the classical inequality
	\begin{equation}
	\label{ineqG}
		\iint_{\R^{2N}} G(|u^H(x) - u^H(y)|) w(|x-y|) \, dx \, dy \le \iint_{\R^{2N}} G(|u(x) - u(y)|) w(|x-y|) \, dx \, dy . 
	\end{equation}
In fact, set 
	\begin{equation*}
		{\mathscr D}^H(u, x,y):=G(u, x^H, y) + G(u, x, y^H) - G(u, x, y) - G(u, x^H, y^H),
	\end{equation*}
	and define 
	\begin{equation*}
		\mathfrak{d}^H(u): = \iint_{\R^{2N}} G(|u^H(x) - u^H(y)|) w(|x-y|) \, dx \, dy - \iint_{\R^{2N}} G(|u(x) - u(y)|) w(|x-y|) \, dx \, dy. 
	\end{equation*}
	As in \eqref{Decomposition-3-1}, we have
	\begin{equation*}
		\mathfrak{d}^H(u)=\int_A \int_B {\mathscr D}^H(x,y;u)
		\Big(w(|x-y|)- w(|x^H- y|) \Big) \, dx \, dy. 
	\end{equation*}
	We claim that 
	\begin{equation}\label{Du-1}
		{\mathscr D}^H(u, x,y) \leq 0,
	\end{equation} 
	for each $x\in A$ and $y\in B$.  Assuming this,  we then immediately get inequality \eqref{frak-in} since $w$ is non-decreasing. We now prove \eqref{Du-1}.  
	Observe that if $a \le b$ and $c \le d$ then 
	$$
	|a -c| + |b-d| \le |d-a| + |b-c|. 
	$$
	and 
	$$
	\max\{ |a -c|,  |b-d| \} \le \max\{ |d-a|, |b-c|\}. 
	$$
	It follows that, for $x \in A$ and $y \in B$,  
	$$
	|u(x) - u(y^H)| + |u(x^H)- u(y)| \le  |u(y) - u(x)|  + |u(x^H) - u(y^H)|
	$$
	and 
	$$
	\max\{ |u(x) - u(y^H)| ,  |u(x^H)- u(y)|\} \le \max\{  |u(y) - u(x)|  , |u(x^H) - u(y^H)|\}. 
	$$
	Assertion~\eqref{Du-1} follows from the properties of Young's functions of $G$. 

As a consequence of \eqref{ineqG}, one has 
\begin{equation*}
	| u^H |_{W^{s, p}(\R^N)} \le | u |_{W^{s, p}(\R^N)}, 
\end{equation*}
for $s \in (0, 1)$ and $p>1$. It follows that, for $u \in W^{s, p}(\R^N)$ with $u>0$, 
\begin{equation*}
	| u^* |_{W^{s, p}(\R^N)} \le | u |_{W^{s, p}(\R^N)}, 
\end{equation*}
where $u^*$ denotes the spherical symmetric rearrangement of $u$. By the BBM formula \eqref{BBM}, one reaches the Polya-Szeg\"o inequality.

\begin{remark}
	\rm 
	Recall that  $\H$ is the set of half-spaces which contain the origin.	One can endow $\H$ with a metric  that ensures that $H_n \to H$ if there exists a sequence of isometries $i_n: \R^N \to \R^N$ such that $H_n = i_n (H)$ and $i_n$ converges to the identity as $n \to + \infty$, moreover, $\H$ is separable with respect to this metric.  Let $\{H_n\}_{n\in\N}$ be a dense set in $\H$. For any $u\in L^p(\R^N)$,  let 
	$\{u_n\}_{n\in\N}$ be the sequence defined by 
	\begin{equation}
	\label{seq}
	u_0:=u \qquad \mbox{ and } \qquad
	u_{n+1}:=u_n^{H_1\cdots H_{n+1}} \mbox{ for } n \ge 0.
	\end{equation}
	Assume that	
	$$
	\sup_{n\in\N} I_{\delta_n}(u_n) < + \infty,\quad \text{for some $\{\delta_n\} \to 0$}.
	$$
	Then $u^* \in W^{1, p}(\R^N)$ ($p>1$), where $u^*$ denotes the Schwarz
	symmetrization of $u$.
	In fact, by \cite[Theorem 1]{jean}, we have $u_n\to u^*$ strongly in $L^p(\R^N)$
	as $n\to\infty$. Then, the assertion follows by
	the Gamma-convergence result in \cite[Theorem 2]{NgGamma}.
	If in addition, for all $\delta>0$,
	$$
	\iint_{\{|u(y)-u(x)|>\delta\}} |x-y|^{-N-p} dxdy
	=\iint_{\{|u^*(y)-u^*(x)|>\delta\}} |x-y|^{-N-p} dxdy,
	$$
	then $u$ is radially symmetric about some point $x_0\in\R^N$
	provided that ${\mathcal L}^N(\{\nabla u^*=0\})=0$.
This follows by the Brothers-Ziemer result \cite[Theorem 1.1]{bzim} jointly with
\cite[Theorem 2]{nguyen06}.
\end{remark}

\section{Radially decreasing functions and $I_\delta$}
\label{sect-compact}

For every measurable function $u:\mathbb{R}^N\to\mathbb{R}$ we define its distribution function
\[
\mu_u(t)=\big|\{x\, :\, |u(x)|>t\}\big|,\qquad t>0.
\] 
Let $0< q<\infty$ and $0<\theta<\infty$, the Lorentz space $L^{q,\theta}(\mathbb{R}^N)$ (cf.\ \cite{lorenz,lorenz2,pick}) is defined by
\[
L^{q,\theta}(\mathbb{R}^N):=\left\{u:\R^N\to\R :\ \int_0^\infty t^{\theta-1}\,\mu_u(t)^{\theta/q}\,dt<\infty\right\}.
\]
In the limit case $\theta=\infty$, this is defined by
\[
L^{q,\infty}(\mathbb{R}^N):=
\left\{u:\R^N\to\R :\ \sup_{t>0} t\,\mu_u(t)^{1/q}<\infty\right\}.
\]
We recall from \cite[Lemma 2.9]{decay} the following 

\begin{lemma}
	\label{lm:border}
	Let $0<\theta\le \infty$ and $0< q<\infty$. Let $u\in L^{q,\theta}(\mathbb{R}^N)$ be a non-negative and radially symmetric decreasing function.
	Then
	\[
	\begin{split}
	0\le u(x)&\le \left(\theta\,\omega_N^{-\frac{\theta}{q}}\,\int_0^\infty t^{\theta-1}\,\mu_u(t)^\frac{\theta}{q}\,dt\right)^\frac{1}{\theta}\, |x|^{-\frac{N}{q}},\qquad \mbox{ if } \theta<\infty,\\
	0\le u(x)&\le \left(\omega_N^{-\frac{1}{q}}\,\sup_{t>0} t\,\mu_u(t)^\frac{1}{q}\right)\,|x|^{-\frac{N}{q}},\qquad \mbox{ if } \theta=\infty.
	\end{split}
	\]
\end{lemma}

\noindent

The next proposition shows that the measure of the superlevels of a nonnegative function 
$u$ is controlled by a quantity involving $I_\delta(u^*)$. Set $p^*:=Np/(N-p)$ for $1 \le p< N$.

\begin{theorem}
	\label{decay}
Let $1 < p  < N$ and $u:\R^N\to\R$ be a non-negative function.\ Then 
there exists two positive constants $C$ and $\lambda$ depending only on $N$ and $p$ such that
\begin{equation}\label{decay-1}
\big|\{x\in\R^N: u(x)> \lambda \delta\}\big|\leq C\delta^{-\frac{Np}{N-p}} \min \big\{ I_\delta(u^*)^{\frac{N}{N-p}}, I_\delta(u^*)^{\frac{N}{N-p}} \big\}, 
\end{equation}
where $u^*$ is the Schwarz symmetric rearrangement of $u$.
Assume that there exists $\vartheta>0$ such that
\begin{equation}
\label{cond}
\int_0^\infty \frac{I_\delta(u)^{\frac{\vartheta}{p}}}{\delta} d\delta<\infty
\qquad\text{or}\qquad
\int_0^\infty \frac{I_\delta(u^*)^{\frac{\vartheta}{p}}}{\delta} d\delta<\infty.
\end{equation}
Then $u,u^* \in L^{p^*,\theta}(\mathbb{R}^N)$
and there exists a positive constant $C$ depending on $u,N,p,\vartheta$ such that
$$
0\le u^*(x)\le C |x|^{-\frac{N-p}{p}}.
$$
\end{theorem}

\begin{proof}
By the definition of $u^*$ we have that $\mu_{u}(t)=\mu_{u^*}(t)$ for $t>0$.
Then, by applying inequality \eqref{Sobolev} to $u^*$, we have, for all $\delta>0$,
$$
\lambda\delta\mu_u(\lambda\delta)^{\frac{N-p}{Np}}=
\lambda\delta\mu_{u^*}(\lambda\delta)^{\frac{N-p}{Np}}
\leq \left(\int_{\{|u^*|>\lambda\delta\}}|u^*|^{\frac{Np}{N-p}}dx\right)^{\frac{N-p}{Np}}\leq C I_\delta(u^*)^{\frac{1}{p}}
$$
and 
$$
\lambda\delta\mu_{u^*}(\lambda\delta)^{\frac{N-p}{Np}}  = \lambda\delta\mu_u(\lambda\delta)^{\frac{N-p}{Np}}
\leq \left(\int_{\{|u^*|>\lambda\delta\}}|u|^{\frac{Np}{N-p}}dx\right)^{\frac{N-p}{Np}}\leq C I_\delta(u)^{\frac{1}{p}}, 
$$
which implies \eqref{decay-1}. By virtue of \eqref{cond},  it follows from \eqref{decay-1} that
$$
\int_0^\infty \delta^{\vartheta-1}\mu_{u}(\delta)^{\frac{\vartheta}{p^*}}  d\delta<\infty \quad \mbox{ and } \quad 
\int_0^\infty \delta^{\vartheta-1}\mu_{u^*}(\delta)^{\frac{\vartheta}{p^*}}  d\delta<\infty,
$$
which yields $u,u^* \in L^{p^*,\theta}(\mathbb{R}^N)$ and the final assertions
follows from Lemma~\ref{lm:border}.
\end{proof}

\begin{remark}\rm
Since $L^{q,q}(\R^N)=L^q(\R^N)$, it follows that $u,u^* \in L^{p^*}(\mathbb{R}^N)$ if
\begin{equation*}
\int_0^\infty \frac{I_\delta(u)^{N/(N-p)}}{\delta} d\delta<\infty
\qquad\text{or}\qquad
\int_0^\infty \frac{I_\delta(u^*)^{N/(N-p)}}{\delta} d\delta<\infty.
\end{equation*}
\end{remark}

%

Concerning the compactness related to $I_\delta$, the following result  was shown in \cite[Theorem 2]{nguyen11}. Let $p > 1$, $(\delta_n) \to 0_+$ and $(u_n) \subset L^p(\R^N)$.  Assume that 
\begin{equation}\label{compact-1}
\{ u_n \}_{n \in \N} \mbox{ is bounded in } L^p(\R^N) \mbox{ and } \{ I_{\delta_n}(u_n) \}_{n \in \N} \mbox{ is bounded}. 
\end{equation}
Then 
\begin{equation} \label{compact-2}
\mbox{$\{u_n \}_{n \in \N}$ is pre-compact in $L^p_{\loc}(\R^N)$}.
\end{equation}

In this paper, we prove the following global compactness result: 
\begin{theorem}
	\label{globalc}
	Let $1<p<N$, $(\delta_n) \to 0_+$,  and let $\{u_n\}_{n\in\N}$ be a sequence of radially symmetric decreasing functions. Assume that 
	$$
	\big\{u_n\big\}_{n \in \N} \mbox{ is bounded in } L^p(\R^N) \mbox{ and } \big\{ I_{\delta_n}(u_n) \big\}_{n \in \N} \mbox{ is bounded}. 
	$$
	Then $\big\{ u_n\big\}_{n\in\N}$ is pre-compact in  $L^r(\R^N)$ for every $p<r<Np/(N-p)$.
\end{theorem}
\begin{proof}
From \eqref{Sobolev}, \eqref{compact-1}, and \eqref{compact-2}, we derive that 
\begin{equation}\label{compact-3}
(u_n) \mbox{ is pre-compact in } L^r_{\loc}(\R^N). 
\end{equation}
Since $(u_n)$ is decreasing and $(u_n)$ is bounded in $L^p(\R^N)$, for any $\delta > 0$ there exists $R_\delta$ such that, for all $n$,  
\begin{equation}\label{compact-part0}
|u_n(x)| \le \delta \mbox{ for } |x| > R_\delta. 
\end{equation}
Fix $\eps > 0$.   By \eqref{compact-3}, there exists a finite subset $J$ of $\N$ such that  
\begin{equation}\label{compact-part1}
\big\{u_n \in L^r (B_{R_\eps}); n \in \N  \big\} \subset \bigcup_{j \in J} \big\{u \in L^r(B_{R_\eps}): \| u - u_j\|_{L^r(B_{R_\eps})} < \eps \big\},  
\end{equation}
where $B_R$ denotes the open ball centered at the origin and of radius $R$ in $\R^N$ for $R>0$. 
On the other hand, by \eqref{compact-part0}, we have 
\begin{equation}\label{compact-part2}
\| u_n\|_{L^r(\R^N \setminus B_{R_\eps})} \le \eps^{(r - p)/r}\| u_n\|_{L^p(\R^N \setminus B_{R_\eps})}^{p/r} \le C \eps^{1 - p/r},  
\end{equation}
since $(u_n)$ is bounded in $L^p(\R^N)$.  A combination of \eqref{compact-part1} and \eqref{compact-part2} yields, for $\eps$ small enough such that $\eps < C \eps^{1 - p/r}$,  
\begin{equation}\label{compact-part3}
\big\{u_n \in L^r (\R^N);  n \in \N  \big\} \subset \bigcup_{j \in J} \big\{u \in L^r(\R^N); \| u -  u_j\|_{L^r(\R^N)} < 3 C \eps^{1- p/r} \big\}, 
\end{equation}
where $C$ is the constant in \eqref{compact-part2}. Since \eqref{compact-part3} holds for small $\eps$ and $1 - r/ p > 0$, it follows that $(u_n)$ is pre-compact in $L^r(\R^N)$. 
\end{proof}

\begin{remark} \rm Let $N \ge 1$ and $p \ge N$. It was shown in \cite[Theorem 1]{nguyen11} that if $u \in L^p(\R^N)$ and $I_\delta (u) < + \infty$ for some $\delta > 0$ then $u \in BMO_{\loc}(\R^N)$ where $BMO$ denotes the space of functions of bounded mean oscillation. By the same proof, under the same assumptions on $\{ u_n \}$,  one obtains the compactness result for $L^r(\R^N)$  for $p< r < p^*$ with $p^* = \infty$. 
\end{remark}

As a consequence of Theorem~\ref{globalc},  we have the following

\begin{corollary}
	Let $1<p<N$ and let $\{u_n\}_{n\in\N}$ be a nonnegative sequence of functions bounded in $L^p(\R^N)$ such that
	$$
	\liminf_{n\to\infty} I_{\delta_n}(u_n^*)<\infty,
	$$
	for some sequence $\{\delta_n\}_{n\in\N} \to0_+$. Then, up to a subsequence, $\{u_n^*\}_{n\in\N}$ converges strongly in $L^r(\R^N)$ to a radial decreasing function, for any $p<r<Np/(N-p)$.
\end{corollary}
\begin{proof}
Since  the sequence $\{u_n\}_{n\in\N}$ is bounded in $L^p(\R^N)$, it follows
that sequence $\{u_n^*\}_{n\in\N}$ is also bounded in $L^p(\R^N)$ by
Cavalieri's principle. Then, by Theorem~\ref{globalc}, it follows that $u_n^*\to v$ in $L^r(\R^N)$ strongly for any $p<r<Np/(N-p)$.
\end{proof}

\section{An open problem for Riesz fractional gradients} \label{sect-PS}
Recently, a notion of {\em fractional gradients} (more precisely 
{\em distributional Riesz fractional gradients}) has been 
introduced in the literature by Shieh and  Spector in the papers \cite{spector1,spector2},
where several basic properties of local Sobolev spaces (e.g.\ Sobolev, Morrey, Hardy, Trudinger inequalities) are proven to extend to fractional spaces defined 
through this new notion. 

More precisely, the fractional gradient $D^s u(x)$ at a point $x\in\R^N$ is defined
for locally Lipschitz compactly supported functions $u:\R^N\to\R$, for any $s\in (0,1)$, by
$$
D^su(x):=c_{N,s}\int_{\R^N}\frac{u(x)-u(y)}{|x-y|^{N+s}}\frac{x-y}{|x-y|}dy,
\qquad\text{for $x\in\R^N$,}
$$
for a suitable positive constant $c_{N,s}$ depending on $N$ and $s$. 
This is reminiscent of the classical scalar notion of $\frac{s}{2}$-{\em fractional laplacian}
$$
(-\Delta)^{s/2}u(x)=C_{N,s}\int_{\R^N}\frac{u(x)-u(y)}{|x-y|^{N+s}}dy,
\qquad\text{for $x\in\R^N$,}
$$
for a suitable normalization constant $C_{N,s}$ depending on $N$ and $s$. 
Notice also that \cite{spector1,spector2}
$$
D^s u=I_{1-s}*Du,\quad u\in C^\infty_c(\R^N), 
\quad\, I_{1-s}(x):=\frac{\gamma(N,s)}{|x|^{N+s-1}}.
$$
According to \cite{spector1,spector2}, one can define, for $p>1$ and $s\in (0,1)$, the space
$$
L^{s,p}(\R^N):=\overline{C^\infty_c(\R^N)}^{\|\cdot\|_{L^p}+\|D^s\cdot\|_{L^p}}.
$$
We now formulate a related open problem.

\begin{open}
	\label{polya-sz-new}
Let $p>1$, $s\in(0,1)$ and $u\in L^{s,p}(\R^N)$ with $u\geq 0$. 
Prove or disprove that $D^su^*\in L^{s,p}(\R^N)$ and the inequality holds
\begin{equation}
\label{Polyagrad}
\int_{\R^N}|D^su^*|^pdx\leq \int_{\R^N}|D^su|^pdx.
\end{equation}
\end{open}

\noindent
In general the inequality 
\begin{equation*}
\int_{\R^N}|D^s |u||^pdx\leq \int_{\R^N}|D^su|^pdx.
\end{equation*}
for all $u\in L^{s,p}(\R^N)$
is {\em not} expected to hold. In particular, one cannot obtain 
inequality \eqref{Polyagrad} for sign-changing functions.
The solution of the above open problem would be very useful in connection
with compact injections for radially symmetric functions of $L^{s,p}(\R^N)$. 
In fact, assume  $N\geq 2$. Notice that, since for any $\varepsilon\in (0,s)$ the injection
$$
L^{s,p}(\R^N)\hookrightarrow W^{s-\eps,p}(\R^N), 
$$
is continuous (cf.\  \cite[(g) of Theorem 2.2]{spector1}) and the injection
$$
W^{s-\eps,p}_{{\rm rad}}(\R^N)\hookrightarrow L^{q}(\R^N),
$$
is compact (cf.\ \cite[Theorem II.1]{lions}) for all $p<q<p^*_{s-\eps}$	it follows that 
the injection
$$
L^{s,p}_{{\rm rad}}(\R^N)\hookrightarrow  L^{q}(\R^N),
$$
is compact for $p<q<Np/(N - sp)$. In particular, nonnegative minimizing sequences for 
$$
u\mapsto \int_{\R^N} |D^su|^pdx,
$$
could be replaced by new minimizing sequences which are radially symmetric 
decreasing and strongly converging in $L^{q}(\R^N)$ for any $p<q<Np/(N - sp)$.


\begin{remark}
	Aiming to prove \eqref{Polyagrad}, with no loss of generality one may assume 
$u\in C^\infty_c(\R^N)$.  Let $H$ be an arbitrary closed given half-space with $0\in H$. 
Define the function $\mathfrak{J}(u):\R^N\to\R$,
$$
\mathfrak{J}(u)(x):=\left|\int_{\R^N}\frac{u(x)-u(y)}{|x-y|^{N+s}}\frac{x-y}{|x-y|}dy\right|^p,
\qquad\text{for all $x\in \R^N$.}
$$	
Then $\mathfrak{J}(u)\in L^1(\R^N)$. 
Setting $v(x):=u(x^H)$ and $w(x):=u^H(x^H)$ for all $x\in\R^N$, we have
	\begin{equation}
	\label{decomposs}
	u^H(x)=v(x)+(u(x)-v(x))^+,\quad
	w(x)=u(x)-(u(x)-v(x))^+,\quad\,\text{for $x\in H$.}
	\end{equation}
	Writing $x^H=x_0+Rx$ where $R$ is a rotation, a change of variable yields
	\begin{align*}
	\mathfrak{J}(u)(x^H)&=\left|\int_{\R^N}\frac{u(x^H)-u(y)}{|x^H-y|^{N+s}}\frac{x^H-y}{|x^H-y|}dy\right|^p \\
	&=\left|\int_{\R^N}\frac{u(x^H)-u(y^H)}{|x^H-y^H|^{N+s}}\frac{x^H-y^H}{|x^H-y^H|}dy\right|^p \\
	&=\left|R\int_{\R^N}\frac{v(x)-v(y)}{|x-y|^{N+s}}\frac{x-y}{|x-y|}dy\right|^p=\mathfrak{J}(v)(x),
	\end{align*}
	and, analogously, $\mathfrak{J}(u^H)(x^H)=\mathfrak{J}(w)(x)$.
	In turn, we conclude that
	\begin{equation*}
		\int_{\R^N}\mathfrak{J}(u)dx =\int_{H}\mathfrak{J}(u)dx+\int_{H}\mathfrak{J}(v)dx,\quad\,
		\int_{\R^N} \mathfrak{J}(u^H)dx=\int_{H} \mathfrak{J}(u^H)dx+\int_{H} \mathfrak{J}(w)dx.
	\end{equation*}
	If one was be able to prove that
	\begin{equation}
	\label{keycond-pol}
	\mathfrak{J}(u^H)(x)+\mathfrak{J}(w)(x)\leq \mathfrak{J}(u)(x)+\mathfrak{J}(v)(x),\quad
	\text{for all $x\in H$,}
	\end{equation}
	then \eqref{Polyagrad} would follow by standard 
	approximations.\ In the local case \eqref{keycond-pol} follows immediately in light of 
	\eqref{decomposs}, while for $\mathfrak{J}$, which is a {\em nonlocal} function, the situation is rather unclear. 
	
	If instead one finds $u$ and $H$
	such that \eqref{keycond-pol} holds with opposite inequality, then the Riesz gradients would already fail the basic polarization inequality.

\end{remark}

\bigskip
\noindent
{\bf Acknowledgements.} The authors would like to warmly thank Daniel Spector
for providing some useful remarks about the content of his works \cite{spector1,spector2}.

\medskip

\end{document}